\documentclass[10pt]{amsart}
\theoremstyle{plain}

\usepackage{amsmath, amsfonts, amsthm, amssymb, amscd, bbold, comment, enumerate}
\usepackage[all,cmtip]{xy}
\usepackage[dvips]{graphics}
\usepackage{epsfig}
\usepackage{color}
\title{A combinatorial proof of the homology cobordism classification of lens spaces}
\author{Margaret Doig}
\address{M. Doig; Syracuse University; Mathematics Department; 215 Carnegie; Syracuse, NY 13244}
\email{midoig@syr.edu}
\author{Stephan Wehrli}
\thanks{S. Wehrli was partially supported by NSF grant DMS-1111680.}
\address{S. Wehrli; Syracuse University; Mathematics Department; 215 Carnegie; Syracuse, NY 13244}
\email{smwehrli@syr.edu}

\theoremstyle{plain}
\newtheorem{theorem}{Theorem}
\newtheorem{corollary}[theorem]{Corollary}
\newtheorem{lemma}[theorem]{Lemma}
\newtheorem{proposition}[theorem]{Proposition}

\newcommand{\Z}{\ensuremath{\mathbb{Z}}}
\newcommand{\Q}{\ensuremath{\mathbb{Q}}}

\begin{document}
\maketitle

\begin{abstract}
It follows implicitly from recent work in Heegaard Floer theory that lens spaces are homology cobordant exactly when they are oriented homeomorphic. We provide a new combinatorial proof using the Heegaard Floer d-invariants, which themselves may be defined combinatorially for lens spaces. 
\end{abstract}

\section*{Introduction}

An integer homology cobordism between two closed, oriented 3-manifolds $Y_1$ and $Y_2$ is a compact, oriented 4-manifold $W$ whose boundary is $\partial W=Y_1\cup -Y_2$ such that the inclusion maps induce isomorphisms $H_i(Y_1;\Z)\cong H_i(W;\Z)\cong H_i(Y_2;\Z)$ for all homology groups; homology cobordism gives an equivalence relation. There are also corresponding definitions of rational homology cobordisms and spin-c rational homology cobordisms.

The homology cobordism classification of the lens spaces was only recently completed. In 1983, Gilmer and Livingston demonstrated that the lens spaces $L(p,q)$ for prime $p$ are homology cobordant iff they are diffeomorphic~\cite{gilmerlivingston}. Fintushel and Stern extended this result in 1988 for odd $p$~\cite{fintushelsternhomologycobordism}. Nicolaescu proved in 2001 that the Ozsv\'ath-Szab\'o d-invariant recovers Reidemeister-Franz torsion~\cite[Section~5]{nicolaescuqhs}, which, in turn, recovers homeomorphism type for lens spaces by results of Brody and Reidemeister~\cite{brody,reidemeister} (technically, Nicolaescu showed the Ozsv\'ath-Szab\'o theta divisor recovers the sum of the Casson-Walker invariant and Reidemeister-Franz torsion, but the Casson-Walker invariant of a lens space is the sum of its d-invariants, by a result of Rasmussen~\cite[Lemma~2.2]{rasmussengodateragaito}, and the theta divisor is the precursor of the d-invariant~\cite{ozsztheta}). In 2011, Greene showed that 2-bridge links are mutants iff their branched double covers (recall, all lens spaces are branched double covers of 2-bridge links) are homeomorphic iff the covers' $\widehat{HF}$ are the same~\cite{greeneconwaymutants}, but $\widehat{HF}$ recovers Reidemeister-Franz torsion by Rustamov~\cite[Theorem~3.4]{rustamov}.

There are many known cobordism invariants, including some from Heegaard Floer homology. Ozsv\'ath and Szab\'o associated the d-invariants to a manifold and spin-c structure which is invariant under spin-c rational homology cobordism, and the d-invariant function on the torsor of spin-c structures is likewise invariant under rational or integral homology cobordism~\cite[Theorem 1.2]{ozszabsolute}. We provide a combinatorial proof that two lens spaces $L(p,q_1)$ and $L(p,q_2)$ share the same d-invariant function precisely when they are oriented homeomorphic. Since the d-invariants are defined combinatorially for lens spaces, this produces a proof of the homology cobordism classification of lens spaces which is entirely combinatorial (modulo the proof that the d-invariants are spin-c homology cobordism invariants; in fact, there is a proof of this invariance for lens spaces which is combinatorial except for its use of Donaldson's Theorem~\cite{greeneconwaymutants}).

\begin{theorem}\label{thm:d-invariants}
Two lens spaces are cobordant by an integral homology cobordism exactly when they are oriented homeomorphic.
\end{theorem}

We begin with a review of facts about d-invariants and spin-c structures and their behavior under homology cobordism. We also define a type of relative d-invariant $f(s,n)$ which carries all the information we need about the d-invariants. Next, we show that, if $Spin^c(L(p,q_1))$ and $d(L(p,q_1),\cdot)$ are isomorphic to $Spin^c(L(p,q_2))$ and $d(L(p,q_2),\cdot)$ in the category of torsors and functions, then $q_1 = q_2$ or $q_1q_2\equiv 1\pmod p$. Finally, we derive a more explicit description of the d-invariants modulo $\Z$ in the special case where $p$ is prime.

\section*{Notation}
Throughout this paper, let $[a]_p$ denote a representative of the class in the interval $[0,p)$. Let $a\equiv_p b$ mean $a$ and $b$ are equivalent modulo $p$. Let $a^\prime$ denote the inverse of $a$ (if it exists), so $aa^\prime\equiv_p 1$.

\section*{Acknowledgements}
Thank you to the many who influenced this paper, including Peter Horn for instigating our investigation of homology cobordism, Adam Levine for initially drawing our attention to Equation~\eqref{eqn:dinvshift}, and Josh Greene for a helpful discussion of the nature of $Spin^c(Y)$ as a torsor.

\section*{d-invariants and spin-c structures}

Heegaard Floer homology assigns several flavors of invariants (including $HF^\infty$ and $HF^+$) to a closed, connected, oriented 3-manifold and a choice of spin-c structure using a Heegaard decomposition of the manifold~\cite{ozsz1,ozsz2}. The generators come with a relative $\Z$-grading. A spin-c cobordism $(W, \mathfrak{s})$ from $(Y_1,\mathfrak{s}|_{Y_1})$ to $(Y_2,\mathfrak{s}|_{Y_2})$ produces a map 
\[
F^+_{W,\mathfrak{s}}: CF^+(Y_1, \mathfrak{s}|_{Y_1}) \longrightarrow CF^+(-Y_2, \mathfrak{s}|_{-Y_2})
\]
which induces a relative grading between generators for the two manifolds:
\begin{equation}\label{eqn:gradingshift}
gr(F^+_{W,\mathfrak{s}}(x)) - gr (x) = \frac{c_1(\mathfrak{s})^2 -2\chi(W)-3sign(W)}{4}.
\end{equation}
For an appropriate choice of spin-c manifold, including a rational homology sphere with its unique spin-c structure, this grading shift allows a lift of the relative $\Z$-grading to an absolute $\Q$-grading by fixing a canonical grading for $S^3$ with its unique spin-c structure. 

Derived from this absolute grading is the \emph{correction term} or \emph{d-invariant} $d(Y,\sigma)$, the minimal grading of any non-torsion element in $HF^+(Y,\sigma)$ inherited from $HF^\infty(Y,\sigma)$~\cite{ozszabsolute}. It is invariant under spin-c rational homology cobordism (if $W$ is a rational homology cobordism, then the right side of Equation~\eqref{eqn:gradingshift} is 0 for both $W$ and $-W$). The d-invariants, as a function on a torsor over $H^2(Y)\cong H_1(Y)$, is also invariant under integral homology cobordism in the following fashion:

\begin{proposition}\label{lem:torsoriso}
If $Y_1$ and $Y_2$ are integrally homology cobordant, then $Spin^c(Y_1)$ and $d(Y_1,\cdot)$ are isomorphic to $Spin^c(Y_2)$ and $d(Y_2,\cdot)$ in the category of torsors and functions. 
\end{proposition}

\begin{proof} 
Let $W$ be the 4-manifold cobordism with $\partial W = Y_1 \cup -Y_2$.

$Spin^c(Y_i)$ is a torsor over $H^2(Y_i)\cong H_1(Y_i)$, and $Spin^c(W)$ s a torsor over $H^2(W)\cong H_2(W,\partial W)$. The long exact sequence for the pair $(W, \partial W)$ splits:
\[\begin{CD}
0 @>>> H_2(W,\partial W) @>r_1^*-r_2^*>> H_1(Y_1)\oplus H_1(-Y_2) @>>> H_1(W) @>>> 0.
\end{CD}\]
where $r_i$ is the restriction map to ${L(p,q_i)}$. This short sequence induces isomorphisms
\[\begin{CD}
H_1(Y_1) @<r_1^*<< H_2(W,\partial W) @>r_2^*>> H_1(Y_2)
\end{CD}\]
 which in turn induce the required torsor isomorphism
\[\begin{CD}
Spin^c(Y_1) @>r_2 r_1^{-1}>> Spin^c(Y_2).
\end{CD}\]
There is a $\Z/2Z$ conjugation action $\mathfrak{t}\mapsto\overline{\mathfrak{t}}$ on the spin-c structures which fixes the spin structures. The restrictions maps and so also this isomorphism respect it.

Because it is invariant under spin-c homology cobordism, $d(Y_1,r_1(t)) = d(Y_2,r_2(t))$, and the functions $d(Y_i,\cdot)$ are isomorphic. 
\end{proof}

The lens space $-L(p,q)$ has a pointed Heegaard diagram $(T^2, \alpha, \beta, z)$ with a single $\alpha$ curve and $\beta$ curve and exactly $p$ intersection points $\alpha \cap \beta$, one in each of the $p$ spin-c structures. For example, the Heegaard decomposition of $-L(5,2)$ looks like:
\begin{center}  
\includegraphics[scale=0.25]{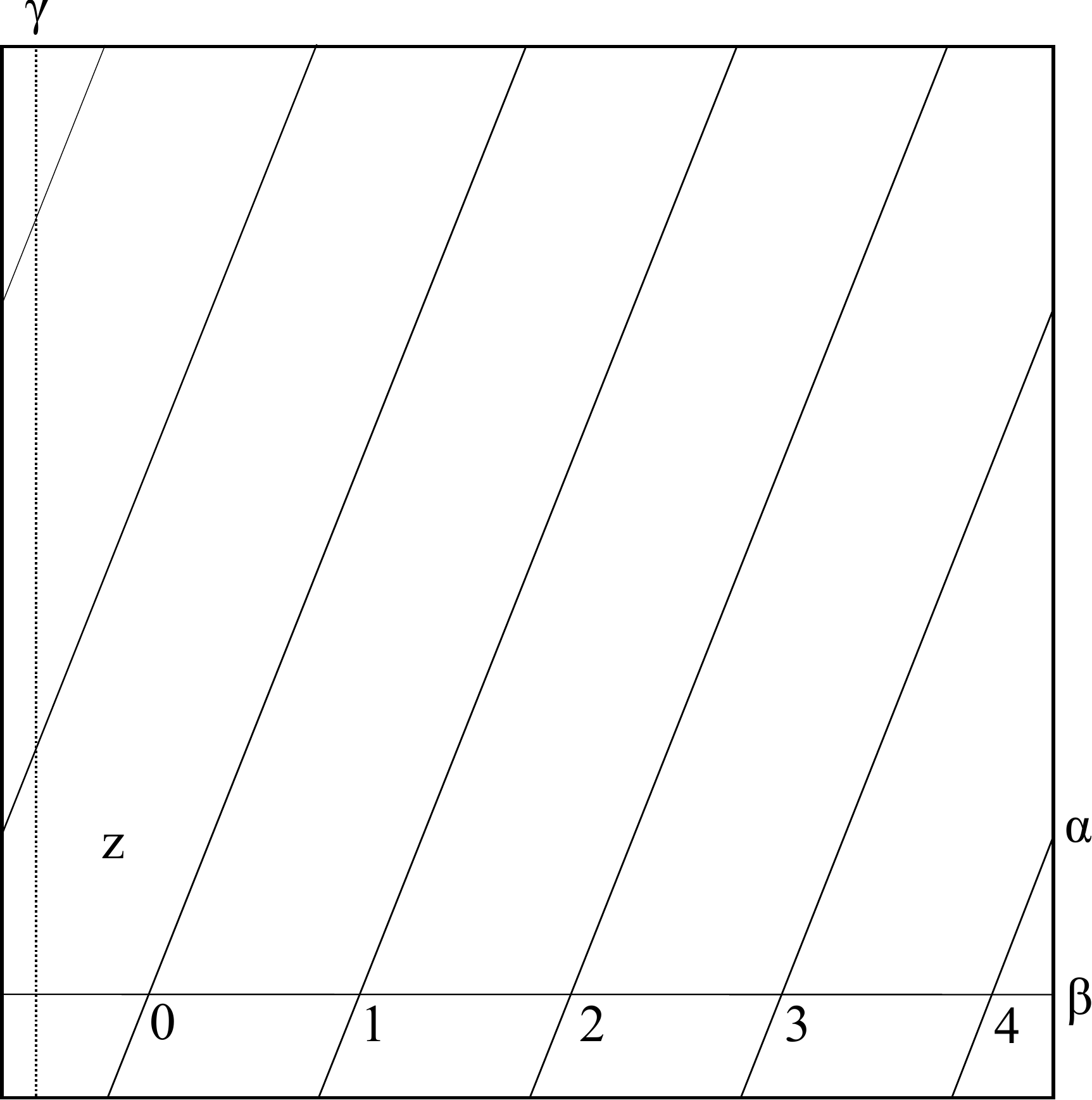}
\end{center}
We have chosen the orientation on $L(p,q)$ so that the manifold is $-p/q$ surgery on the unknot. Choose an identification of $Spin^c(L(p,q))$ by labelling the intersection points $0, 1, \dots, p-1$ from left to right across the bottom of the diagram, beginning with the $0$ for the bottom right corner of the domain containing the basepoint $z$~\cite[Proposition~4.8]{ozszabsolute}. To see the difference of two spin-c structures $i-j\in H_1(L(p,q))$ under this identification, observe the curve $\gamma$, which is a generator of $H_1(L(p,q))$ and connects $i$ to $i+q$ along the $\alpha$ curve and $i+q$ to $i$ along the $\beta$ curve, so we say $(i+q)-i=[\gamma]$. Any other $i-j$ gives a multiple of $[\gamma]$.

There is a combinatorial description of the d-invariants of a lens space based on the grading shift in Equation~\eqref{eqn:gradingshift} derived in~\cite[Proposition~4.8]{ozszabsolute}. Assuming $0<q<p$,
\[
d(L(p,q),i) = \frac{1}{4}- \frac{(2[i]_p+1-p-q)^2}{4pq} - d(L(q,p),i).
\]
Derived from this recursive formula is a more direct formula for how the d-invariants change under the $\gamma$-action~\cite[Corollary~5.2]{leelipshitz}:
\begin{equation}\label{eqn:dinvshift}
d(L(p,q),i+q)-d(L(p,q),i)=\frac{p-1-2[i]_p}{p}.
\end{equation}
The spin structures are exactly the integers among the following: 
 \begin{equation}\label{eqn:spin}
 \frac{q-1}{2} \quad \mathrm{and} \quad \frac{p+q-1}{2}.
 \end{equation}
This result may be deduced from Equation~\eqref{eqn:dinvshift}: The conjugation action which fixes a spin structure $s$ must identify $s+n$ and $s-n$, and $d(L(p,q),i+q)=d(L(p,q),i-q)$ implies $\frac{p-1-2i}{p}=-\frac{p-1-2(i-q)}{p}$, or $2i\equiv_p q-1$. For alternative explanations, Cf~\cite[p. 134]{ueozsz} or \cite[Lemma 6.1]{cochranhorn}. Note that both numbers in \eqref{eqn:spin} are axes of symmetry, and both are integers when $p$ is even, but only one is an integer (and so a spin structure) when $p$ is odd.

In the case of two homology cobordant lens spaces $L(p,q_i)$, Proposition~\ref{lem:torsoriso} also implies:
\[
d(L(p,q_1), s_1+a) = d(L(p,q_2), r_2 r_1^{-1}(s_1) + (r_2 r_1^{-1})^*(a)) = d( s_2+ua).
\]
for any $a\in H_1(L(p,q_1))$, where $s_1$ and $s_2$ are chosen spin structures which are restrictions of a common spin structure on $W$. The last equality follows because $(r_2 r_1^{-1})^*$ is an isomorphism of $\Z/p\Z$, which means it is multiplication by some unit $u\in \Z/p\Z$.

\section*{Relative d-invariants}
Let $p$ and $q$ be coprime with $p>q>0$. Choose a spin structure $s$ as in \eqref{eqn:spin} (if $p$ is odd, this choice is forced). We renormalize the d-invariants of $L(p,q)$ by defining a function $f(s,\cdot)\colon\Z/p\Z\rightarrow\Z$ using this choice of spin structure:
\[
f(s,n):=pd(L(p,q),s+nq)-pd(L(p,q),s).
\]

\begin{lemma}\label{lem:propertiesoff}
The function $f$ obeys
\[\begin{split}
f(s,0) &=0\\
f(s,n+1) &= f(s,n)+p-1-2\left[s+nq\right]_p\\ 
f(s,n) &\equiv_p-n^2q 
\end{split}\]
If $L(p,q_1)$ and $L(p,q_2)$ are homology cobordant by $W$, and if $f_1$ and $f_2$ are the corresponding functions for some compatible choice of spin structure $s_1$ and $s_2$ which restrict the same spin structure on $W$,
\[\begin{split}
f_2(s_2,n) &= f_{1}(s_1,nu)\\
q_2 &\equiv_p u^2 q_1
\end{split}\]
for some unit $u\in \Z/p\Z$.
\end{lemma}
A cobordism between two lens spaces tells us about the torsor structure defined above.

\begin{proof}
The first two equalities follow from  Equation~\eqref{eqn:dinvshift} and the definitions of $f$ and $s$. The third equality holds because $f(s,n+1)\equiv_p f(s,n) -(2n+1)q$ and $f(s,0)\equiv_p 0$. The fourth 
follows from Proposition~\ref{lem:torsoriso}, assuming that the spin structures were chosen so that $r_2r_1^{-1}(s_1) = s_2$, and the last equality follows from the third and fourth.
\end{proof}

\section*{Proof of Theorem~\ref{thm:d-invariants}}
We will now prove the main theorem. 

\begin{proof}[Proof of Theorem~\ref{thm:d-invariants}] By Lemma~\ref{lem:propertiesoff}, there is a unit $u$ such that 
\[
q_1= q\quad\mbox{and}\quad q_2\equiv_p u^2q.
\] 
There is also some choice of spin structures $s_1$ and $s_2$ which are restrictions of the same spin structure on $W$. Define $g\colon\Z/p\Z\longrightarrow\Z$ such that
\[
g(m):=f_1(s_1,m q^\prime)=f_2(s_2,mu^\prime q^\prime),
\]
which is well-defined by the definitions of $f_i(s_i,m)$ and Lemma~\ref{lem:propertiesoff}.

Apply the recursive equations for $f_i(s_i,n)$ from Lemma~\ref{lem:propertiesoff} to $f_1(s_1,m q^\prime + 1)$ and $f_2(s_2,mu^\prime q^\prime+1)$ to see that $g$ satisfies the relations:
\[
g(m+q)=g(m)+p-1-2\left[s_1+m\right]_p
\]
and
\[
g(m+uq)=g(m)+p-1-2\left[s_2+mu\right]_p.
\]

Since the above relations must hold for all $m$, we can compute $g(m+uq+q)$ in two ways, as $g((m+uq)+u)$ or as $g((m+u)+uq)$. Since the results must be the same, we get
\begin{equation}\label{eqn:key_identity}
\left[s_1+m\right]_p+\left[s_2+(m+q)u\right]_p
=\left[s_2+mu\right]_p+\left[s_1+m+uq\right]_p.
\end{equation}
Now recall that
\[
\left[X+Y\right]_p
=\begin{cases}
[X]_p+[Y]_p  & \mbox{if $[X]_p<p-[Y]_p$},\\
[X]_p+[Y]_p-p & \mbox{if $[X]_p\geq p-[Y]_p$.} 
\end{cases}
\] 
Equation~\eqref{eqn:key_identity} is therefore equivalent to the condition that
\begin{equation}\label{eqn:key_equivalence}
\left[s_1+m\right]_p<p-\left[uq\right]_p \iff \left[s_2+mu\right]_p<p-\left[uq\right]_p
\end{equation}
for all $m\in\Z/p\Z$.

By Lemma~\ref{lem:countingalot} below, Condition \eqref{eqn:key_equivalence} can only be satisfied for all $m\in\Z/p\Z$ if either $u\equiv_p \pm 1$ or $uq\equiv_p \pm 1$. That is, either $q_2=q_1$ or $q_1q_2\equiv_p 1$.
\end{proof}

Note that we did not use any information in the above proof about the explicit forms the $s_i$ take, merely the fact that there exist $s_1$ and $s_2$ which are restrictions of some spin structure on $W$; in particular, the parity of $p$ is irrelevant.

We now address two technical lemmata required for the proof above. 

\begin{lemma}\label{lem:counting}
Let
\[
H: \{0, 1, \cdots, p-1\} \longrightarrow \{0, 1, \cdots, p-1\}
\]
be a function such that $H(i) \equiv_p H(0)+in.$ If
\[
H(i) < C \iff i < C
\]
where $2 \leq C \leq p-2$, then
\[
H(i) = i \quad \mathrm{or} \quad H(i) = C-1-i.
\]
\end{lemma}

A few experiments will quickly convince the reader that this lemma should be true. For the sake of completeness, we prove:

\begin{proof} Choose $-p/2\leq n \leq p/2$. Assume, for the moment, that $C\leq p/2$. 

If $n=1$, then $H(0)=0$ and $H(i) = i$. 

If $n=-1$, then $H(0) = C-1$ and $H(i) = C-i-1$.

For any other $n$, there will eventually be an $i<C$ with $H(i)\geq C$. For example, for $n \geq 2$, take 
\[
i_0 = \left\lfloor\frac{C-H(0)}{n}\right\rfloor+1
\]
Note that $0 < i_0 <C$ since $C\geq 2$, and
\[
0<H(0) + ni_0 \leq H(0) + n\left(\frac{C-H(0)}{n}+1\right)=C+n\leq p
\]
so we may remove the $\equiv_p$ in the definition of $H(i_0)$:
\[
H(i_0) = H(0) + ni_0\geq H(0)+n\left(\frac{C-H(0)}{n}\right) = C,
\]
as desired.

Similarly, for $n\leq-2$, take 
\[i_0 =  \left\lfloor\frac{H(0)}{|n|}\right\rfloor+1.\]
Now $-p<H(0) + ni_0< 0$, so 
\[
H(i_0) = H(0) + ni_0 + p \geq H(0) - |n| \left(\frac{H(0)}{|n|}+1\right)+p\geq p - |n| \geq p/2\geq C
\]

It is easy to adjust the above proof to accommodate $C \geq p/2$. The key is that some (at least two and at most $p-2$) adjacent values of $i$ map to (the same number of) adjacent values of $H(i)$. The following are equivalent:
\[\begin{split}
H(i) < C &\iff i<C\\
0 \leq H(i)\leq C-1 &\iff 0\leq i \leq C-1 \\
p-C \leq H(i)+p-C\leq p-1&\iff p-C\leq i+p-C \leq p-1 \\
p-C \leq H(i-p+C)+p-C\leq p-1&\iff p-C\leq i \leq p-1 \\
H(i+C)-C<p-C &\iff i < p-C
\end{split}\]
and $\widehat{H}(i)=H(i+C)-C$ also obeys the rule $\widehat{H}(i) \equiv_p \widehat{H}(0)+in.$
\end{proof}

\begin{lemma}\label{lem:countingalot}
Let
\[\begin{split}
f(m) &= [x+my]_p\\
F(m) &= [X+mY]_p.
\end{split}\]
with $y$ and $Y$ units modulo $p$. If
\[
f(m) < C \iff F(m) < C
\]
for some $2 \leq C \leq p-2$, then
\[
Y=\pm y.
\]
\end{lemma}

\begin{proof}
Rescale $m$ by precomposing $f$ and $F$ with
\[
m(i) = (i-x)y^\prime.
\]
Then
\[
h(i) := f(m(i)) = [i]_p
\]
and 
\[
H(i):= F(m(i)) = [(X-xy'Y)+i(y^\prime Y)]_p.
\]

The lemma statement is equivalent to 
\[
h(i) < C \iff H(i) < C,
\]
which is equivalent to
\[
H(i) < C \iff i <C.
\]
Note that $H(i) \equiv_p H(0)+in$, and apply Lemma~\ref{lem:counting}.

If $H(i)=i$, then $y^\prime Y=1$, or $Y=y$, and $X-xy'Y=0$, or $X=x$.  

If $H(i)=C-1-i$, then $H(0)=C-1\equiv_p X-xy^\prime Y$ and $H(C-1)=0\equiv_p X-xy^\prime Y+(C-1)y^\prime Y \equiv_p C-1+(C-1)y^\prime Y$, so $Y\equiv_p-y$ and $X\equiv_p-x+C-1$.
\end{proof}

\section*{If $p$ is prime}

In the special case where $p$ is a prime, we have a more precise description of the d-invariants modulo \Z. Consider the reduction of $f$ modulo $p$, $\overline{f}(s,\cdot)\colon\Z/p\Z\rightarrow\Z/p\Z$. We denote by
\[
\overline{S}(L(p,q))\subseteq\Z/p\Z
\]
the image of $\overline{f}$.

\begin{theorem}\label{thm:main} Let $p$ be prime number and $q$ coprime to $p$.
\begin{enumerate}
\item[(a)]
If $q$ is a quadratic residue modulo $p$, then
\[
\overline{S}(L(p,q))=\{a\in\Z/p\Z\,|\,\mbox{$-a$ is a square in $\Z/p\Z$}\}.
\]
\item[(b)]
If $q$ is a quadratic non-residue modulo $p$, then
\[
\overline{S}(L(p,q))=\{a\in\Z/p\Z\,|\,\mbox{$-a$ is not a square in $\Z/p\Z$}\}\cup\{0\}.
\]
\end{enumerate}
\end{theorem}

In the residue case, a more explicit description of the d-invariants is possible: 
\begin{corollary}\label{cor:main}Let $p$ be an odd prime number and $q$ a residue coprime to $p$.
\begin{enumerate}
\item[(a)]
There is only one $n$ such that $\overline{f}(s,n)=0$, namely, $n=0$.
\item[(b)]
For every $a\in\overline{S}(L(p,q))\setminus\{0\}$, there are exactly two $n$ such that $\overline{f}(s,n)=a$.
\item[(c)]
$\overline{S}(L(p,q))$ contains exactly $(p+1)/2$ elements.
\end{enumerate}
\end{corollary}

\begin{proof}[Proof of Theorem~\ref{thm:main}] 
Since $f(s,n)\equiv_p-n^2q$,
\[
\overline{S}(L(p,q))=\{a\in\Z/p\Z\,|\,\mbox{$a$ satisfies $a\equiv_p -n^2q$ for some $n$}\}.
\]

If $a=0$, then $n=0$. 

Let $\left(\tfrac{m}{p}\right)$ denote the Legendre symbol of $m$ and $p$, defined by
\[
\left(\frac{m}{p}\right):=\begin{cases}
1&\mbox{if $m$ is a quadratic residue modulo $p$},\\
-1&\mbox{if $m$ is a quadratic non-residue modulo $p$},\\
0&\mbox{if $m$ is zero modulo $p$}.
\end{cases}
\]
Assume $a\neq 0$. Then the condition $a\equiv_p-n^2q$ can be written as $-aq^\prime\equiv_p n^2$, or 
\[
\left(\frac{-aq^\prime}{p}\right)=1.
\]
Since the Legendre symbol is multiplicative in the first argument, we can write the condition as
\[
\left(\frac{-a}{p}\right)\left(\frac{q^\prime}{p}\right)=1,
\]
and, multiplying both sides by $\left(\tfrac{q}{p}\right)$, we get
\[
\left(\frac{-a}{p}\right)=\left(\frac{q}{p}\right),
\]
where we have used that $\left(\tfrac{q^\prime}{p}\right)\left(\tfrac{q}{p}\right)=\left(\tfrac{qq^\prime}{p}\right)=\left(\tfrac{1}{p}\right)=1$.
We can thus write $\overline{S}(L(p,q))$ as
\[
\overline{S}(L(p,q))=\left\{a\in\Z/p\Z\,\Big|\,\mbox{$a=0$ or $\left(\tfrac{-a}{p}\right)=\left(\tfrac{q}{p}\right)$}\right\}.
\]
\end{proof}

\begin{proof}[Proof of Corollary~\ref{cor:main}] If $p$ is prime, $(\Z/p\Z)[x]$ is a unique factorization domain. If $p \neq 2$, this means every equation $n^2\equiv_p-aq^\prime$ with $a\neq 0$ has exactly two solutions. Part (c) follows because the total number of d-invariants, counted with multiplicities, is equal to $p$.
\end{proof}

Note that $\overline{S}(L(2,1))=\Z/2\Z$, so (b) and (c) are false for $p=2$.

\bibliographystyle{amsalpha}
\bibliography{/Users/mdoig/Documents/math/biblio-MASTER}
\end{document}